\documentclass[reqno]{amsart}
\usepackage{amssymb}
\usepackage{amsfonts}
\usepackage{amsthm}
\usepackage{amsmath}
\usepackage[hypertex]{hyperref}

\newcommand{\R}{{\mathbb{R}}}

\newcommand{\ntx}{{\nabla_{\!t,x}}}

\newcommand{\eps}{{\varepsilon}}

\DeclareMathOperator*{\supp}{supp}

\newcommand{\sqrtDelta}{|\nabla|}  

\newcommand{\qtq}[1]{\quad\text{#1}\quad}

\newcommand{\attn}[1]{\marginpar{\footnotesize\raggedright\cite{ATTN}\kern -2.2ex\vrule width 2ex depth 0.5ex height 2ex\relax\,#1}}

\newtheorem{theorem}{Theorem}[section]

\newtheorem{lemma}[theorem]{Lemma}

\newtheorem{corollary}[theorem]{Corollary}

\newtheorem{proposition}[theorem]{Proposition}
\theoremstyle{definition}
\newtheorem{definition}[theorem]{Definition}

\theoremstyle{remark}

\newtheorem*{remarks}{Remarks}

\newcounter{smalllist}

\numberwithin{equation}{section}
\allowdisplaybreaks

\begin{document}

\title[The radial defocusing energy-supercritical NLW]{The radial defocusing energy-supercritical nonlinear wave equation
in all space dimensions}
\author{Rowan Killip}
\address{University of California, Los Angeles}
\author{Monica Visan}
\address{University of California, Los Angeles}

\begin{abstract}
We consider the defocusing nonlinear wave equation $u_{tt}-\Delta u + |u|^p u=0$ with spherically-symmetric initial data in the regime $\frac4{d-2}<p<\frac4{d-3}$ (which is energy-supercritical) and dimensions $3\leq d\leq 6$; we also consider $d\geq 7$, but for a smaller range of $p>\frac4{d-2}$.  The principal result is that blowup (or failure to scatter)
must be accompanied by blowup of the critical Sobolev norm.  An equivalent formulation is that maximal-lifespan solutions with bounded critical Sobolev norm are global and scatter.
\end{abstract}

\date{\today}

\maketitle


%
%
%
%

\section{Introduction}

We consider the initial value problem for the defocusing nonlinear wave equation in $d\geq3$ space dimensions:
\begin{equation}\label{nlw}
\begin{cases}
\ u_{tt} - \Delta u + F(u) = 0\\
\ u(0)=u_0,  \ u_t(0)=u_1,
\end{cases}
\end{equation}
where the nonlinearity $F(u)=|u|^p u$ is energy-supercritical, that is, $p>\frac{4}{d-2}$.

The class of solutions to \eqref{nlw} is left invariant by the scaling
\begin{equation}\label{scaling}
u(t,x)\mapsto \lambda^\frac2{p} u(\lambda t,\lambda x),
\end{equation}
which determines the \emph{critical} Sobolev space for initial data, namely, $(u_0,u_1)\in\dot H^{s_c}_x\times\dot H^{s_c-1}_x$
where the \emph{critical regularity} is $s_c:=\frac d2-\frac 2p$.  Notice that $p>\frac{4}{d-2}$ precisely corresponds to $s_c>1$; since the energy
\begin{equation}\label{energy}
E(u) = \int_{\R^d} \tfrac12|u_t|^2 + \tfrac12|\nabla u|^2 + \tfrac{1}{p+2}|u|^{p+2} \, dx
\end{equation}
scales like $s=1$, this regime is known as \emph{energy-supercritical}.

Let us start by making the notion of a solution more precise.

\begin{definition}[Solution]\label{D:solution}
A function $u: I \times \R^d \to \R$ on an open  time interval $0\in I \subset \R$ is a \emph{(strong) solution} to
\eqref{nlw} if $(u,u_t)\in C^0_t (K ;\dot H^{s_c}_x\times\dot H^{s_c-1}_x)$ and $u\in L_{t,x}^{(d+1)p/2}(K \times \R^d)$ for
all compact $K \subset I$, and obeys the Duhamel formula
\begin{equation}\label{old duhamel}
\begin{aligned}
\begin{bmatrix} u(t)\\[1ex] u_t(t)\end{bmatrix}
&= \begin{bmatrix} \cos(t\sqrtDelta) &  \sqrtDelta^{-1} \sin(t\sqrtDelta) \\[1ex]
            -\sqrtDelta\sin(t\sqrtDelta) & \cos(t\sqrtDelta) \end{bmatrix}
                \begin{bmatrix} u(0) \\[1ex] u_t(0)\end{bmatrix} \\
&\qquad - \int_{0}^{t} \begin{bmatrix} |\nabla|^{-1}\sin\bigl((t-s)\sqrtDelta\bigr) \\[1ex]
        \cos\bigl((t-s)\sqrtDelta\bigr) \end{bmatrix} F(u(s))\,ds
\end{aligned}
\end{equation}
for all $t \in I$.  We refer to the interval $I$ as the \emph{lifespan} of $u$. We say that $u$ is a \emph{maximal-lifespan
solution} if the solution cannot be extended to any strictly larger interval. We say that $u$ is a \emph{global solution} if $I=\R$.
\end{definition}

We define the \emph{scattering size} of a solution to \eqref{nlw} on a time interval $I$ by
\begin{equation}\label{E:S defn}
\|u\|_{S(I)}:= \Bigl(\int_I \int_{\R^d} |u(t,x)|^{\frac{(d+1)p}2}\, dx \,dt\Bigr)^{\frac2{(d+1)p}}.
\end{equation}

Associated to the notion of solution is a corresponding notion of blowup.  By the standard local theory, the following precisely corresponds to the impossibility of continuing the solution
in a manner consistent with Definition~\ref{D:solution}.

\begin{definition}[Blowup]\label{D:blowup}
We say that a solution $u$ to \eqref{nlw} \emph{blows up forward in time} if there exists a time $t_1 \in I$ such that
$$ \|u\|_{S([t_1, \sup I))} = \infty$$
and that $u$ \emph{blows up backward in time} if there exists a time $t_1 \in I$ such that
$$ \|u\|_{S((\inf I, t_1])}= \infty.$$
\end{definition}

Our purpose here is to give a short proof of the following result:

\begin{theorem}[Spacetime bounds]\label{T:main}
Assume that $\frac4{d-2}< p<\frac4{d-3}$ for $3\leq d\leq 6$ and that $\frac4{d-2}< p<\frac{d(d-1)-\sqrt{d^2(d-1)^2-16(d+1)^2}}{2(d+1)}$ if $d\geq 7$.
Let $u:I\times\R^d\to\R$ be a spherically-symmetric solution to \eqref{nlw} such that $(u,u_t) \in L_t^\infty (I; \dot H^{s_c}_x \times \dot H^{s_c-1}_x)$.  Then
$$
\|u\|_{S(I)}\leq C\bigl(\|(u, u_t)\|_{L_t^\infty (I; \dot H^{s_c}_x \times \dot H^{s_c-1}_x)}\bigr).
$$
\end{theorem}

For $d=3$ this was proved by Kenig and Merle \cite{kenig-merle:wave sup}.  In \cite{OK} we proved this result for non-radial data.  While preparing \cite{OK}, we realized that it is possible to give a short proof of Theorem~\ref{T:main} that works uniformly in all dimensions.  That is the topic of this paper.  In keeping with our goal of a simple presentation, we have restricted ourselves to the specific values of $p$ stated in the theorem.  These hypotheses represent the combination of two restrictions, one related to the local theory (which dominates in high dimensions) and another dictated by the Morawetz inequality ($s_c<3/2$).  The former restriction stems from our desire to present as simple and uniform a local theory as possible.  While one may certainly obtain a larger range of $p$ in this setting by some piecemeal approach, it is not clear to us how to obtain the full range dictated by our principal hypothesis $s_c<3/2$.  Indeed, note that one natural restriction is the smoothness condition $s_c<p+1$; this allows us to take $s_c$ derivatives of the nonlinearity.  Our condition for $d\geq 7$ is equivalent to
$s_c<p+1 -(\frac1{d+1} +\frac p2)$.

In low dimensions, the sole restriction on $p$ is $p<\frac{4}{d-3}$ and corresponds to $s_c < 3/2$, which nevertheless covers all values of $p$ when $d=3$.  This condition is dictated by the Morawetz inequality.  The methods presented here do not immediately extend to higher values of $p$.  The problem arises in Section~\ref{S:no soliton} and could be circumvented by proving that the almost periodic solutions discussed below actually lie in $L^\infty_t(\dot H_x^s\times \dot H_x^{s-1})$ for some $s<3/2$.  Arguments showing how this can be done (even in the non-radial setting) may be found in \cite{Berbec,Miel,OK}; however, this is significantly more involved than what we chose to present here.

As mentioned above, finite-time blowup of a solution to \eqref{nlw} must be accompanied by divergence of the scattering
size defined in \eqref{E:S defn}.  Thus, Theorem~\ref{T:main} immediately implies

\begin{corollary}[Spacetime bounds]\label{C:main} If $u:I\times\R^d\to\R$ is a maximal-lifespan spherically-symmetric solution to \eqref{nlw} with
$(u,u_t)\in L_t^\infty (I; \dot H^{s_c}_x \times \dot H^{s_c-1}_x)$, then $u$ is global,
$$
\|u\|_{S(\R)}\leq C\bigl(\|(u, u_t)\|_{L_t^\infty (\R; \dot H^{s_c}_x \times \dot H^{s_c-1}_x)}\bigr),
$$
and $u$ scatters in the sense that
$$
\bigl \|  (u,u_t) - (u^\pm, u_t^\pm)\bigr\|_{\dot H^{s_c}_x\times\dot H^{s_c-1}_x} \to 0 \quad \text{as} \quad t\to \pm\infty,
$$
for two solutions $u^\pm$ of the linear wave equation.
\end{corollary}

This corollary takes on a more appealing form if we rephrase it in the contrapositive:

\begin{corollary}[Nature of blowup]\label{C:blowup} A spherically-symmetric solution $u:I\times\R^d\to\R$ to \eqref{nlw} can only blow up in finite time
or be global but fail to scatter if its $\dot H^{s_c}_x \times \dot H^{s_c-1}_x$ norm diverges.
\end{corollary}

When $p=\frac 4{d-2}$, or equivalently, $s_c=1$, the critical Sobolev norm is automatically bounded in time by virtue of the
conservation of energy.  This \emph{energy-critical} case of \eqref{nlw} has received particular attention because of this property.  Global
well-posedness was proved in a series of works \cite{Grillakis:3Dwave,Grillakis:lowDwave,Kapitanskii,Rauch,Struwe:rad,ShatahStruwe:reg,ShatahStruwe}
with finiteness of the scattering size being added later; see \cite{bahouri-gerard,GinibreSofferVelo,Nakanishi,Pecher,tao:Ecrit wave}.  Certain monotonicity formulae, namely
the Morawetz and energy flux identities, play an important role in all these results.  It is important that these monotonicity formulae also have critical scaling.

In the energy-supercritical case discussed in this paper, all conservation laws and monotonicity formulae have scaling below the critical regularity.
At the present moment, there is no technology for treating large-data dispersive equations without some \emph{a priori} control of a critical norm.
This is the purpose of the $L^\infty_t(I; \dot H^{s_c}_x \times \dot H^{s_c-1}_x)$ assumption in Theorem~\ref{T:main}; it plays the role of the missing
conservation law at the critical regularity.  To deal with the fact that the basic monotonicity formula scales like the energy (rather than the critical regularity),
we employ a space truncation in the manner of \cite{borg:scatter}; see also \cite{KVZ:quadratic, tao:gwp radial}.

\subsection{Outline of the proof}\label{SS:outline}

We argue by contradiction.  By the fundamental observations of Keraani \cite{keraani-l2} and Kenig--Merle \cite{KenigMerle:H1}, we know that failure of Theorem~\ref{T:main}
guarantees the existence of certain minimal counterexamples and moreover, such solutions have good compactness properties.  These properties are best described in terms of
the following notion:

\begin{definition}[Almost periodicity modulo scaling]\label{D:ap}
A solution $u$ to \eqref{nlw} with lifespan $I$ is said to be \emph{almost periodic modulo scaling} if $(u,u_t)$ is bounded
in $\dot H^{s_c}_x \times \dot H^{s_c-1}_x$ and there exist functions $N: I \to \R^+$ and $C: \R^+ \to \R^+$ such that for all $t \in I$ and $\eta > 0$,
$$
\int_{|x| \geq C(\eta)/N(t)} \bigl||\nabla|^{s_c} u(t,x)\bigr|^2\, dx
+ \int_{|x| \geq C(\eta)/N(t)} \bigl||\nabla|^{s_c-1} u_t(t,x)\bigr|^2\, dx\leq \eta
$$
and
$$
\int_{|\xi| \geq C(\eta) N(t)} |\xi|^{2s_c}\, | \hat u(t,\xi)|^2\, d\xi
+ \int_{|\xi| \geq C(\eta) N(t)} |\xi|^{2(s_c-1)}\, | \hat u_t(t,\xi)|^2\, d\xi\leq \eta .
$$
We refer to the function $N(t)$ as the \emph{frequency scale function} for the solution $u$ and to $C(\eta)$ as the \emph{compactness modulus function}.
\end{definition}

\begin{remarks}
1. Spherical symmetry forces the bulk of the solution to concentrate around the spatial origin.  This is the reason for the absence of a spatial center function $x(t)$.

2. The continuous image of a compact set is compact.  Thus, by Sobolev embedding, almost periodic (modulo scaling)
solutions obey the following: For each $\eta>0$ there exists $C(\eta)>0$ so that
\begin{equation}\label{E:u' compact}
\bigl\| u(t,x) \bigr\|_{L^\infty_t L_x^{\frac{dp}{2}}(\{|x|\geq C(\eta)/N(t)\})}
    + \bigl\| \ntx u(t,x) \bigr\|_{L^\infty_t L_x^{\frac{dp}{p+2}}(\{|x|\geq C(\eta)/N(t)\})}
    \leq \eta,
\end{equation}
where $\ntx u =(u_t,\nabla u)$ denotes the space-time gradient of $u$.
\end{remarks}

With these preliminaries out of the way, we can now describe the first major milestone in the proof of Theorem~\ref{T:main}.

\begin{theorem}[Three special scenarios for blowup]\label{T:enemies}
Suppose that Theorem~\ref{T:main} failed.  Then there exists a maximal-lifespan spherically-symmetric solution $u:I\times\R^d\to \R$, which
obeys $(u,u_t) \in L_t^\infty (I; \dot H^{s_c}_x \times \dot H^{s_c-1}_x)$, is almost periodic modulo scaling, and
$\|u\|_{S(I)}=\infty$.  Moreover, we can also ensure that the lifespan $I$ and the frequency scale function $N:I\to\R^+$
match one of the following three scenarios:
\begin{itemize}
\item[I.] (Finite-time blowup) We have that either $\sup I<\infty $ or $|\inf I|<\infty$.
\item[II.] (Soliton-like solution) We have $I = \R$ and $N(t) = 1$ for all $t \in \R$.
\item[III.] (Low-to-high frequency cascade) We have $I = \R$,
\begin{equation*}
\inf_{t \in \R} N(t) \geq 1, \quad \text{and} \quad \limsup_{t \to +\infty} N(t) = \infty.
\end{equation*}
\end{itemize}
\end{theorem}

The proof of Theorem~\ref{T:enemies} follows a well-travelled path.  See \cite{ClayNotes} for an introduction to these techniques including two worked examples and references up to that time.
Let us briefly review the ingredients: (i) A concentration compactness principle (= profile decomposition) for the linear propagator.  The very first example of this was worked out
for the wave equation (with $d=3$) in \cite{bahouri-gerard}.   The extension to all dimensions can be found in \cite{Aynur}.  (ii) A perturbation theory for the nonlinear equation.
While the basic framework is standard, each equation has its peculiarities, particularly when small-power nonlinearities are involved.  We discuss this at some length in Section~\ref{S:stability}, in part because
our arguments unify and simplify existing results for certain special cases.  (iii) A decoupling argument.  This is usually fairly direct; however, some subtleties arise in the model discussed in this paper
due to the fact that $s_c>1$ and $p$ is small.  The requisite technology can be found in \cite{Miel}.

With Theorem~\ref{T:enemies} in hand, the proof of Theorem~\ref{T:main} reduces to showing that none of the three special scenarios can occur.  In Section~\ref{S:ftb} we show that the first
scenario, a finite-time blowup solution, cannot exist because it is inconsistent with the conservation of energy.  In Section~\ref{S:no soliton} we show that neither of the other two
scenarios can occur, since they are inconsistent with the (truncated) Morawetz identity (cf. Lemma~\ref{L:Morawetz}) when $p<\frac4{d-3}$.

\subsection*{Acknowledgements}
The first author was supported by NSF grant DMS-0701085.  The second author was supported by NSF grant DMS-0901166.

%
%
%
%

\section{NLW background}

We write $X\lesssim Y$ to indicate that $X\leq C Y$ for some dimension-dependent constant $C$, which may change from line to line.  Other dependencies will be indicated with subscripts, for example, $X\lesssim_u Y$.

\subsection{Strichartz estimates}
One of the most fundamental tools in the modern analysis of nonlinear wave equations is the Strichartz estimate.  We record some particular
instances of this estimate below.  For further information, see \cite{tao:keel,Pecher, Strichartz77} and the references therein.

\begin{definition}[Admissible pairs]
We say that the pair $(q,r)$ is \emph{wave-admissible} if
$$
\tfrac1q + \tfrac{d-1}{2r}\leq \tfrac{d-1}4, \quad 2\leq q\leq \infty, \qtq{and} 2\leq r<\infty.
$$
\end{definition}

\begin{lemma}[Strichartz estimates]\label{L:Strichartz}
Fix $d\geq 3$.  Let $I$ be a compact time interval and let $u: I\times\R^d \to \mathbb \R$ be a solution to the forced wave equation
$$
u_{tt}- \Delta u + F = 0.
$$
Then for any $t_0\in I$ and any wave-admissible pair $(q,r)$,
\begin{align*}
\bigl\|\ntx u\bigr\|_{L_t^\infty \dot H^{s_c-1}_x}&+ \|u\|_{S(I)} + \bigl\| |\nabla|^{s_c-\frac12}u\bigr\|_{L_{t,x}^{\frac{2(d+1)}{d-1}}}
+ \bigl\||\nabla|^{\gamma-1} \ntx u\|_{L_t^q L_x^r}\\
&\lesssim \bigl\|\ntx u(t_0)\bigr\|_{\dot H^{s_c-1}_x}+ \bigl\| |\nabla|^{s_c-\frac12} F\bigr\|_{L_{t,x}^{\frac{2(d+1)}{d+3}}},
\end{align*}
provided $\frac1q+\frac dr=\frac 2p + \gamma$.  All spacetime norms in the formula above are on $I\times\R^d$.
\end{lemma}

We will use the notation
\begin{align}\label{Ssc}
\|u\|_{S^{s_c}(I)}:= \sup \bigl\||\nabla|^{\gamma-1} \ntx u\|_{L_t^q L_x^r},
\end{align}
where the supremum is taken over all admissible pairs $(q,r)$ and numbers $\gamma$ obeying the scaling condition $\frac1q+\frac dr=\frac 2p + \gamma$, with $r\leq r_*(d,p)$ dictated by the largest exponent appearing
in the arguments below.

The following result will be needed in Section~\ref{S:stability}; its proof requires only minor modifications to the proof of the Christ--Weinstein fractional chain rule presented
in \cite[\S2.5]{ToolsPDE}.

\begin{lemma}[Derivatives of differences]\label{L:frac rule}
Let $F(u)=|u|^pu$ with $p>0$ and let $0<s<1$.  Then for $1<q,q_1,q_2<\infty$ such that $\tfrac1q=\tfrac 1{q_1} + \tfrac p{q_2}$, we have
$$
\bigl\| |\nabla|^s [F(u+v) -F(u)]\bigr\|_q\lesssim \bigl\| |\nabla|^s u\bigr\|_{q_1} \|v\|_{q_2}^p + \| |\nabla|^s v\bigr\|_{q_1} \|u+v\|_{q_2}^p.
$$
\end{lemma}

\begin{proof}
By the Fundamental Theorem of Calculus,
\begin{align*}
\bigl|F(u(x))- F(u(y))\bigr| &= |u(x)-u(y)| \biggl| \int_0^1 F'\bigl(u(y) + t[u(x)-u(y)]\bigr)\, dt\biggr|\\
&\lesssim_p |u(x)-u(y)| \bigl\{ |u(x)|^p + |u(y)|^p \bigr\},
\end{align*}
and similarly for $F(u+v)$.  Combining these gives
\begin{align*}
\bigl| [F(u+v)-F(u)](x)&- [F(u+v)-F(u)](y) \bigr| \\
&\lesssim \bigl\{ |v(x)-v(y)| + |u(x)-u(y)| \bigr\}\bigl\{ H(x) + H(y) \bigr\},
\end{align*}
where $H(x)=|u(x)|^p + |v(x)|^p$.  With this estimate in hand, one may now follow verbatim the arguments used to prove Proposition~5.1 in \cite[\S2.5]{ToolsPDE}.
More precisely, this estimate is used in place of (5.4) in that book.
\end{proof}

As $p$ may be less than one, we will also need a version of the fractional chain rule for fractional powers:

\begin{lemma}[Fractional chain rule for a H\"older continuous function, \cite{Monica:thesis}]\label{L:FDFP}
Suppose $G$ is a H\"older continuous function of order $0<p<1$.  Then, for every $0<s<p$, $1<q<\infty$,
and $\tfrac{s}p<\sigma<1$ we have
\begin{align*}
\bigl\| |\nabla|^s G(u)\bigr\|_q
\lesssim \bigl\||u|^{p-\frac{s}{\sigma}}\bigr\|_{q_1} \bigl\||\nabla|^\sigma u\bigr\|^{\frac{s}{\sigma}}_{\frac{s}{\sigma}q_2},
\end{align*}
provided $\tfrac{1}{q}=\tfrac{1}{q_1} +\tfrac{1}{q_2}$ and $(1-\frac s{p \sigma})q_1>1$.
\end{lemma}

\subsection{Morawetz inequality}
Next we derive a space-localized Morawetz identity (cf. \cite{Morawetz75,MorawetzStrauss}).  It is a very close
analogue of the formula used in \cite{borg:scatter} in the NLS setting; see also \cite{KVZ:quadratic,tao:gwp radial}.

\begin{lemma}[Space-localized Morawetz]\label{L:Morawetz}
Let $u:I\times \R^d\to\R$ be a solution of \eqref{nlw} with $d\geq 3$.  Then
\begin{align*}
\int_I \int_{|x|\leq |I|} \frac{|u(t,x)|^{p+2}}{|x|} \,dx \,dt \lesssim |I|^{d-2-\frac4p} \Bigl(\|\ntx u\|^2_{L^\infty_t \dot H_x^{s_c-1}} +\|\ntx u\|_{L^\infty_t \dot H_x^{s_c-1}}^{p+2} \Bigr).
\end{align*}
\end{lemma}

\begin{proof}
By direct computation, one sees that taking the time derivative of
\begin{equation*}
M(t):=\int_{\R^3} -a_j(x)u_t(t,x) u_j(t,x) - \tfrac12 a_{jj}(x) u(t,x) u_t(t,x) \,dx
\end{equation*}
yields
\begin{equation*}
\frac{dM}{dt} = \int_{\R^3} a_{jk}(x) u_j(t,x)u_k(t,x) + \tfrac{p}{2(p+2)} a_{jj}(x) u(t,x)^{p+2} - \tfrac14a_{jjkk}(x) u(t,x)^2 \,dx
\end{equation*}
for general functions $a:\R^d\to\R$.  Here subscripts denote spatial derivatives and repeated indices are summed.  Setting $R=|I|$, we choose $a(x):=R\, \psi(|x|/R)$,
where $\psi(r)$ is a smooth non-decreasing function obeying $\psi(r)=r$  if $r\leq 1$ and $\psi(r)=3/2$ when $r\geq 2$.

Simple computations show that for $|x|\leq R$, we have
\begin{align*}
a_k(x)&=\frac{x_k}{|x|},\quad  a_{jj}(x) = \frac{d-1}{|x|}> 0,\quad -a_{jjkk}>0 \quad \text{(as a distribution),}
\end{align*}
and the matrix $a_{jk}(x)$ is positive definite.  On the other hand, when $R\le |x|\le 2R$,
\begin{align*}
|a_k(x)| \lesssim 1,\quad |a_{jk}(x)|&\lesssim R^{-1},\qtq{and}\ |-a_{jjkk}(x)|\lesssim R^{-3},
\end{align*}
while all derivatives of $a$ vanish when $|x|\geq 2R$.  Combining this information with H\"older's, Hardy's, and Sobolev's inequalities yields
$$
| M(t) | \lesssim R^{d-2-\frac4p} \|\ntx u\|^2_{L^\infty_t L_x^{\frac{dp}{p+2}}} \lesssim R^{d-2-\frac4p} \|\ntx u\|^2_{L^\infty_t\dot H_x^{s_c-1}}
$$
and similarly,
$$
\frac{dM}{dt}  \geq \tfrac{p(d-1)}{2(p+2)} \!\! \int_{|x|\leq R} \!\! \frac{|u(t,x)|^{p+2}}{|x|} \,dx - R^{d-3-\frac4p}O\Bigl(\|\ntx u\|^2_{L^\infty_t \dot H_x^{s_c-1}} +\|\ntx u\|_{L^\infty_t \dot H_x^{s_c-1}}^{p+2} \Bigr).
$$
The result now follows by the Fundamental Theorem of Calculus.
\end{proof}

\subsection{Potential energy concentration}

To make use of Lemma~\ref{L:Morawetz}, we need a lower bound on the left-hand side.  This cannot be done pointwise in time, but we do have the following:

\begin{proposition}[Potential energy concentration] \label{P:pot conc}
Let $u$ be a global solution to \eqref{nlw} that is almost periodic modulo scaling.  Then, there exists $C=C(u)$ so that
\begin{equation}\label{E:some pot E}
\int_I \int_{|x-x(t)|\leq C/ N(t)} N(t) |u(t,x)|^{p+2} \,dx\,dt  \gtrsim_u \int_I N(t)^{\frac4p-(d-3)} \,dt,
\end{equation}
uniformly for all intervals $I=[t_1,t_2]\subseteq \R$ with $t_2\geq t_1 + N(t_1)^{-1}$.
\end{proposition}

\begin{proof}
Without the additional factor of $N(t)$ on the left (and so also on the right), this is proved in \cite{OK}; see Corollary~3.5 there.
However, the very first step in that proof is to split $I$ into intervals where $N(t)$ is essentially constant.  For this reason,
\eqref{E:some pot E} also follows from the argument presented there.
\end{proof}

%
%
%
%

\section{Stability}\label{S:stability}
In this section we sketch the proof of the stability result for \eqref{nlw}, the only part of the proof of Theorem~\ref{T:enemies} that is not already in the literature.
We note here that the proof works equally well in the defocusing and focusing cases.

\begin{theorem}[Stability]\label{T:stab}
Assume that  $\frac4{d-2}\leq p<\frac4{d-3}$ if $3\leq d\leq 6$ and assume that $\frac4{d-2}\leq p<\frac{d(d-1)-\sqrt{d^2(d-1)^2-16(d+1)^2}}{2(d+1)}$ if $d\geq7$.
Let $I$ a compact time interval containing zero and let $\tilde u$ be an approximate solution to \eqref{nlw} on $I\times \R^d$ in the sense that
$$
\tilde u_{tt} -\Delta \tilde u + F(\tilde u)+ e=0
$$
for some function $e$.  Assume that
\begin{align}
\bigl\|\ntx \tilde u\bigr\|_{L_t^\infty \dot H_x^{s_c-1}(I\times\R^d)}&\le M \label{finite Sob norm}\\
\|\tilde u\|_{S(I)}&\le L \label{finite st norm}
\end{align}
for some positive constants $M$ and $L$.  Let $(u_0, u_1)\in \dot H_x^{s_c}\times \dot H_x^{s_c-1}$ be such that
\begin{align}
\bigl\| (u_0, u_1)-(\tilde u_0, \tilde u_1) \bigr\|_{\dot H_x^{s_c}\times \dot H_x^{s_c-1}}&\le \eps \label{close initially}
\end{align}
and suppose also that the error $e$ obeys
\begin{align}
\bigl\| |\nabla|^{s_c-1/2} e \bigr\|_{L_{t,x}^{\frac{2(d+1)}{d+3}}(I\times\R^d)}&\le \eps \label{error small}
\end{align}
for some $0<\eps<\eps_1=\eps_1(M,L)$. Then, there exists a unique strong solution $u:I\times\R^d\mapsto \R$ to \eqref{nlw} with initial data
$(u_0, u_1)$ at time $t=0$; moreover,
\begin{align}
\|u-\tilde u\|_{S(I)} &\leq C(M,L)\eps^{c} \label{close in st norm}\\
\bigl\| u-\tilde u\bigr\|_{S^{s_c}(I)} &\leq C(M,L)\label{close in Strich}\\
\bigl\| u \bigr\|_{S^{s_c}(I)} &\leq C(M,L) \label{u in Sc},
\end{align}
where $c$ is a positive constant that depends on $d,p,M,$ and $L$.
\end{theorem}

The general strategy for proving stability for a dispersive equation is by now standard and is reviewed along with references in \cite{ClayNotes}.  Indeed, special cases of Theorem~\ref{T:stab}
have appeared before; see \cite{BCLPZ,kenig-merle:wave, kenig-merle:wave sup}.  Nevertheless, there is some flexibility in the method in terms of which spaces one chooses to work in
and we contend that we provide a simpler treatment of the existing results just mentioned.

As in our previous work on the nonlinear Schr\"odinger equation \cite{ClayNotes, Miel}, we will work in spaces with a small fractional number of derivatives.  We close our bootstrap in the following spaces:
\begin{align*}
\|u\|_{X(I)}:=\bigl\||\nabla|^{s_c-\frac12}u\bigr\|_{L_{t,x}^{\frac{2(d+1)}{d-1}}(I\times\R^d)} \qtq{and}
\|F\|_{Y(I)}:=\bigl\||\nabla|^{s_c-\frac12}F\bigr\|_{L_{t,x}^{\frac{2(d+1)}{d+3}}(I\times\R^d)}
\end{align*}
when $3\leq d\leq 6$ and when $d\geq 7$,
\begin{align*}
\|u\|_{X(I)}&:=\bigl\||\nabla|^{p/2} u\bigr\|_{L_t^{\frac{2p(d+1)}{4(d+1) +p^2(d+1)-pd(d-1)}}L_x^{\frac{2(d+1)}{d-1}}}\\
\|F\|_{Y(I)}&:=\bigl\||\nabla|^{p/2} F\bigr\|_{L_t^{\frac{2p(d+1)}{4(d+1) +p^2(d+1)-p(d^2-d-4)}}L_x^{\frac{2(d+1)}{d+3}}(I\times\R^d)}.
\end{align*}
The additional restriction on $p$ when $d\geq 7$ is to ensure that the time exponent in the definition of $X$ (and so also $Y$) is positive (and finite).

The space $X(I)$ in which the solution will be measured is related to the space $Y(I)$ in which the nonlinearity will be measured via a Strichartz-type inequality:

\begin{lemma}[A Strichartz-type inequality]\label{L:Strich XY}
\begin{align*}
\Bigl\| \int_0^t \tfrac{\sin( (t-s)|\nabla|)}{|\nabla|} F(s)\, ds\Bigr\|_{X(I)}\lesssim \|F\|_{Y(I)}.
\end{align*}
\end{lemma}

\begin{proof}
In dimensions $3\leq d\leq 6$, this is an instance of the usual Strichartz inequality (cf. Lemma~\ref{L:Strichartz}), while for $d\geq 7$ it is one of the standard exotic Strichartz estimates.  The proof is simple
and the same in either case; we review it below.

As noted for example in \cite[\S4.3]{ShatahStruwe}, the wave propagator obeys the frequency-localized dispersive estimate
$$
\Bigl\|\tfrac{\sin( (t-s)|\nabla|)}{|\nabla|}P_N f \Bigr\|_{L_x^{\frac{2(d+1)}{d-1}}} \lesssim |t-s|^{-\frac{d-1}{d+1}}\|P_N f\|_{L_x^{\frac{2(d+1)}{d+3}}},
$$
from which the result follows by the Hardy--Littlewood--Sobolev inequality and elementary Littlewood--Paley theory.
\end{proof}

We also note the following relations between the various spaces involved:

\begin{lemma}\label{L:connections}
With $S$ as in \eqref{E:S defn} and $S^{s_c}$ as in \eqref{Ssc},
\begin{align*}
\|u\|_{X(I)}&\leq \|u\|_{S^{s_c}(I)}\\
\|u\|_{S(I)} &\leq \|u\|_{X(I)}^{ \theta}\|u\|_{S^{s_c}(I)}^{1- \theta} \qtq{for some} 0<\theta<1\\
\|u\|_{X(I)}&\leq \|u\|_{S(I)} ^{ \tilde\theta} \|u\|_{S^{s_c}(I)}^{1-\tilde\theta} \qtq{for some} 0< \tilde\theta<1\\
\|F(u)\|_{Y(I)}&\lesssim \|u\|_{X(I)}\|u\|_{S(I)}^p\lesssim  \|u\|_{X(I)}^{1+ \theta p}\|u\|_{S^{s_c}(I)}^{(1-\theta)p}.
\end{align*}
Our last estimates are for differences of nonlinearities.  For $3\leq d\leq 6$,
\begin{align}\label{non 1}
\|F(u)-&F(\tilde u)\|_{Y(I)}\\
&\lesssim  \|\tilde u\|_{X(I)} \|u-\tilde u\|_{S(I)}^{p} + \|u-\tilde u\|_{X(I)} \bigl\{ \|\tilde u\|_{S(I)}^{p} + \| u- \tilde u\|_{S(I)}^{p} \bigr\} ,\notag
\end{align}
while for $d\geq 7$ we need
\begin{align}\label{non 2}
\|F(u)-&F(\tilde u)\|_{Y(I)}\\
&\lesssim \|u-\tilde u\|_{X(I)} \bigl\{\|u-\tilde u\|_{S^{s_c}(I)}^{p-\frac pd}\|u-\tilde u\|_{S(I)}^{\frac pd}
+\|\tilde u\|_{S^{s_c}(I)}^{p-\frac pd}\|\tilde u\|_{S(I)}^{\frac pd}\bigr\},\notag
\end{align}
as well as the direct analogue of \eqref{non 1}, namely,
\begin{align}
\bigl\| |\nabla|^{s_c-\frac12}&[F(u)-F(\tilde u)]\bigr\|_{L_{t,x}^{\frac{2(d+1)}{d+3}}(I\times\R^d)} \label{non 3}\\
&\lesssim \|\tilde u\|_{S^{s_c}(I)}\| u-\tilde u\|_{S(I)}^p+ \|u-\tilde u\|_{S^{s_c}(I)}\bigl\{ \|\tilde u\|_{S(I)}^p + \|u-\tilde u\|_{S(I)}^p\bigr\}.\notag
\end{align}
\end{lemma}

\begin{proof}
The first four estimates follow from straightforward applications of Sobolev embedding, interpolation, H\"older's inequality, together with the fractional chain rule Lemma~\ref{L:FDFP}.

The inequalities \eqref{non 1} and \eqref{non 3} are consequences of Lemma~\ref{L:frac rule} and H\"older's inequality.

To derive \eqref{non 2}, one first uses the Fundamental Theorem of Calculus to write
$$
F(u)-F(\tilde u) = (u-\tilde u)\int_0^1 F'( \tilde u + \tau [u-\tilde u])\, d\tau
$$
and thence, via the fractional product rule (and Sobolev embedding),
$$
\|F(u)-F(\tilde u)\|_{Y(I)}
\lesssim \|u-\tilde u\|_{X(I)} \sup_{\tau\in[0,1]} \big\| |\nabla|^{p/2} F'( \tilde u + \tau [u-\tilde u]) \bigr\|_{L_t^{\frac{d+1}2} L_x^{\frac{2d(d+1)}{4d+p(d+1)}}}.
$$
The estimate \eqref{non 2} now follows from Lemma~\ref{L:FDFP}, in which we take $\sigma=\frac d{2(d-1)}$.
\end{proof}

With these preliminaries out of the way, we are ready to resume the proof of Theorem~\ref{T:stab}.   For the remainder of the proof, all spacetime norms are over $I\times\R^d$.
By standard iterative arguments (and subdividing the original time interval), it suffices to prove the claim with hypothesis \eqref{finite st norm} replaced by
\begin{align} \label{small st norm}
\|\tilde u\|_{S(I)}&\le \delta,
\end{align}
where $\delta$ is a small constant allowed to depend on $M$.  By Lemma~\ref{L:connections}, we see that one can transfer bounds (and smallness) between the $X$ and $S$ norms;
thus, it suffices to prove the claim with the norm $S$ replaced by the norm $X$ in \eqref{close in st norm}.

Next, an application of Lemma~\ref{L:Strichartz} along with \eqref{finite Sob norm}, \eqref{error small}, and \eqref{small st norm} yields
\begin{align*}
\|\tilde u\|_{S^{s_c}(I)}
&\lesssim \bigl\|\ntx \tilde u\bigr\|_{L_t^\infty \dot H_x^{s_c-1}} +\bigl\| |\nabla|^{s_c-\frac12} e \bigr\|_{L_{t,x}^{\frac{2(d+1)}{d+3}}}
+\bigl\| |\nabla|^{s_c-\frac12}F(\tilde u)\bigr\|_{L_{t,x}^{\frac{2(d+1)}{d+3}}}\\
&\lesssim M + \eps + \delta^p\|\tilde u\|_{S^{s_c}(I)}.
\end{align*}
Hence, for $\delta$ sufficiently small depending on $d, p$ and $\eps$ small enough depending on $M$,
\begin{align}\label{tilde u in Ssc}
\|\tilde u\|_{S^{s_c}(I)}\lesssim M.
\end{align}

We first explain how to complete the argument in the case when $3\leq d\leq 6$.  Note that in this case, the power $p$ under discussion satisfies $p\geq 1$.
An application of the Strichartz inequality Lemma~\ref{L:Strichartz} along with Lemma~\ref{L:connections}, \eqref{close initially}, \eqref{error small}, \eqref{small st norm},
and \eqref{tilde u in Ssc} yields
\begin{align*}
\|u-\tilde u\|_{S(I)} &+ \|u-\tilde u\|_{X(I)}\\
&\lesssim \bigl\| (u_0, u_1)- (\tilde u_0,\tilde u_1) \bigr\|_{\dot H_x^{s_c}\times \dot H_x^{s_c-1}} + \bigl\| |\nabla|^{s_c-\frac12} e \bigr\|_{L_{t,x}^{\frac{2(d+1)}{d+3}}}\\
&\quad+\bigl\| |\nabla|^{s_c-\frac12}[F(u)-F(\tilde u)]\bigr\|_{L_{t,x}^{\frac{2(d+1)}{d+3}}}\\
&\lesssim \eps +\|\tilde u\|_{X(I)} \|u-\tilde u\|_{S(I)}^{p} + \|u-\tilde u\|_{X(I)} \bigl\{ \|\tilde u\|_{S(I)}^{p} + \| u- \tilde u\|_{S(I)}^{p} \bigr\}\\
&\lesssim \eps + \delta^{\tilde \theta} M^{1-\tilde \theta} \|u-\tilde u\|_{S(I)}^{p} +\delta^p \|u-\tilde u\|_{X(I)} +\|u-\tilde u\|_{X(I)}\| u- \tilde u\|_{S(I)}^{p}.
\end{align*}
Taking $\delta$ small depending on $d,p,M$ and recalling that $p\geq 1$, a bootstrap argument yields
$$
\|u-\tilde u\|_{S(I)} + \|u-\tilde u\|_{X(I)}\lesssim \eps,
$$
which settles \eqref{close in st norm}.  Another application of the Strichartz inequality as above proves \eqref{close in Strich}.  The estimate \eqref{u in Sc} follows from the
triangle inequality, \eqref{close in Strich}, and \eqref{tilde u in Ssc}.

We now consider the case when $d\geq 7$.  Arguing as in the main estimate for the lower dimensional case, but now using Lemma~\ref{L:Strich XY} as well, we obtain
\begin{align*}
\|u-&\tilde u\|_{X(I)}\\
&\lesssim \bigl\| (u_0, u_1)- (\tilde u_0,\tilde u_1) \bigr\|_{\dot H_x^{s_c}\times \dot H_x^{s_c-1}} + \bigl\| |\nabla|^{s_c-\frac12} e \bigr\|_{L_{t,x}^{\frac{2(d+1)}{d+3}}} +\|F(u)-F(\tilde u)\|_{Y(I)}\\
&\lesssim \eps + \|u-\tilde u\|_{X(I)} \bigl\{\|u-\tilde u\|_{S^{s_c}(I)}^{p- \frac pd}\|u-\tilde u\|_{S(I)}^{\frac pd}
+\|\tilde u\|_{S^{s_c}(I)}^{p- \frac pd}\|\tilde u\|_{S(I)}^{\frac pd}\bigr\}\\
&\lesssim \eps + \|u-\tilde u\|_{X(I)}^{1+\theta \frac pd} \|u-\tilde u\|_{S^{s_c}(I)}^{p-\theta \frac{p}d}
+\delta^{\frac pd} M^{p- \frac{p}d}\|u-\tilde u\|_{X(I)},
\end{align*}
which for $\delta$ small (depending on $d, p, M$) yields
\begin{align}\label{1}
\|u-\tilde u\|_{X(I)}\lesssim \eps + \|u-\tilde u\|_{X(I)}^{1+\theta \frac pd} \|u-\tilde u\|_{S^{s_c}(I)}^{p-\theta \frac{p}d}.
\end{align}
Another application of the same considerations (using \eqref{tilde u in Ssc}) yields
\begin{align*}
\bigl\|u-&\tilde u\bigr\|_{S^{s_c}(I)}\\
&\lesssim  \bigl\| (u_0, u_1)- (\tilde u_0,\tilde u_1) \bigr\|_{\dot H_x^{s_c}\times \dot H_x^{s_c-1}}+ \bigl\| |\nabla|^{s_c-\frac12} e \bigr\|_{L_{t,x}^{\frac{2(d+1)}{d+3}}} \\
&\quad +\bigl\| |\nabla|^{s_c-\frac12}[F(u)-F(\tilde u)]\bigr\|_{L_{t,x}^{\frac{2(d+1)}{d+3}}}\\
&\lesssim \eps + \|\tilde u\|_{S^{s_c}(I)}\| u-\tilde u\|_{S(I)}^p +  \|u-\tilde u\|_{S^{s_c}(I)}\bigl\{\|u-\tilde u\|_{S(I)}^p+\|\tilde u\|_{S(I)}^p\bigr\}\\
&\lesssim \eps + \Bigl[M \|u-\tilde u\|_{S^{s_c}(I)}^{(1-\theta) p}+ \|u-\tilde u\|_{S^{s_c}(I)}^{1+(1-\theta) p}\Bigr]\| u-\tilde u\|_{X(I)}^{\theta p}  +\delta^p\|u-\tilde u\|_{S^{s_c}(I)},
\end{align*}
which for $\delta$ small (depending on $d, p$) gives
\begin{align}\label{2}
\|u-\tilde u\|_{S^{s_c}(I)}\lesssim \eps + \Bigl[M \|u-\tilde u\|_{S^{s_c}(I)}^{(1-\theta) p}+ \|u-\tilde u\|_{S^{s_c}(I)}^{1+(1-\theta) p}\Bigr]\| u-\tilde u\|_{X(I)}^{\theta p}.
\end{align}

A simple bootstrap argument using \eqref{1} and \eqref{2} yields \eqref{close in st norm} and \eqref{close in Strich}.  The claim \eqref{u in Sc} follows by the triangle inequality, \eqref{close in Strich}, and \eqref{tilde u in Ssc}.
This completes the proof of the theorem.

%
%
%
%

\section{The global solutions}\label{S:no soliton}

In this section we preclude the soliton-like and frequency-cascade solutions described in Theorem~\ref{T:enemies}.

\begin{theorem}[Absence of solitons and frequency-cascades]
There are no soliton-like or frequency-cascade solutions to \eqref{nlw} in the sense of Theorem~\ref{T:enemies}.
\end{theorem}

\begin{proof}
We argue by contradiction.  Assume there exists a solution $u:\R\times\R^d\to \R$ that is either a soliton-like or a frequency-cascade solution in the sense
of Theorem~\ref{T:enemies}.  We will show these scenarios are inconsistent with the Morawetz inequality.

By Lemma~\ref{L:Morawetz},
\begin{align}\label{Mor}
\int_0^T \int_{|x|\leq T} \frac{|u(t,x)|^{p+2}}{|x|}\, dx\, dt\lesssim_u T^{d-2-\frac4p},
\end{align}
for any $T>0$.  On the other hand, by Proposition~\ref{P:pot conc} we have concentration of potential energy, that is,
there exists $C=C(u)$ so that for any $T\geq N(0)^{-1}\geq 1$,
$$
\int_0^T \int_{|x|\leq C/N(t)}N(t)  |u(t,x)|^{p+2}\, dx\, dt \gtrsim_u \int_0^T N(t)^{\frac4p-(d-3)}\, dt.
$$
Using the fact that $N(t)\geq 1$ and $p<\frac4{d-3}$, we obtain that for $T\geq 1$,
\begin{align*}
\text{LHS\eqref{Mor}}
& \geq \int_0^T \int_{|x|\leq C/N(t)} \frac{|u(t,x)|^{p+2}}{|x|}\, dx\, dt
\gtrsim_u \int_0^T N(t)^{\frac4p-(d-3)}\, dt
\gtrsim_u T.
\end{align*}
Choosing $T$ sufficiently large depending on $u$ and recalling once again that $p<\frac4{d-3}$ (and hence $\frac4p-(d-2)<1$), we derive a contradiction to \eqref{Mor}.
\end{proof}

%
%
%
%

\section{The finite-time blowup solution}\label{S:ftb}

In this section, we preclude the finite-time blowup scenario described in Theorem~\ref{T:enemies} by showing that such solutions
are inconsistent with the conservation of energy.

\begin{theorem}[Absence of finite-time blowup solutions]
There are no finite-time blowup solutions to \eqref{nlw} in the sense of Theorem~\ref{T:enemies}.
\end{theorem}

\begin{proof}
We argue by contradiction.  Assume there exists a solution $u:I\times\R^d\to \R$ that is a finite-time blowup solution
in the sense of Theorem~\ref{T:enemies}.  By the time-reversal and time-translation symmetries, we may assume that
the solution blows up as $t\searrow 0=\inf I$.

First note that $N(t)\to\infty$ as $t\to0$, for otherwise any subsequential limit of the functions $N(t)^{-\frac2p} u(t, N(t)^{-1}x)$ would
blow up instantaneously, in contradiction of the local theory.

Next we show that
\begin{align}\label{bounded support 1}
\supp u(t) \cup \supp u_t(t) \subseteq B(0, t) \quad  \text{for all} \quad t\in I,
\end{align}
where $B(0,t)$ denotes the closed ball in $\R^d$ centered at the origin of radius $t$.  Indeed, it suffices to show that
\begin{align}\label{bounded support 2}
\lim_{t\to 0}\int_{t +\eps \leq |x|\leq \eps^{-1} - t} \tfrac12 \bigl| \ntx u(t,x)\bigr|^2 + \tfrac1{p+2} |u(t,x)|^{p+2} \, dx = 0
\quad \text{for all} \quad \eps>0,
\end{align}
because the energy on the annulus $\{x:t +\eps \leq |x|\leq \eps^{-1} - t\}$ is finite and does not decrease as $t\to0$.  To obtain \eqref{bounded support 2}, fix $\eps>0$.
As $N(t)\to\infty$ as $t\to0$, we deduce that for all $\eta>0$ there exists $t_0=t_0(\eps, \eta)$ such that for $0<t<t_0$ we have
$$
\{x\in \R^d:\ t +\eps \leq |x|\leq \eps^{-1} - t\} \subseteq \{x\in \R^d:\ |x|\geq C(\eta)/N(t)\},
$$
where $C(\eta)$ is as in \eqref{E:u' compact}.  Thus by H\"older's inequality and \eqref{E:u' compact},
\begin{align*}
&\int_{t +\eps \leq |x|\leq \eps^{-1} - t}   \tfrac12 \bigl| \ntx u(t,x)\bigr|^2 + \tfrac1{p+2} |u(t,x)|^{p+2} \, dx\\
&\quad\lesssim \eps^{\frac4{p}-(d-2)} \Bigl[\bigl\|\ntx u(t)\bigr\|_{L_x^{\frac{dp}{p+2}}(\{|x|\geq C(\eta)/N(t)\})}^2
    + \|u(t)\|_{L_x^{\frac{dp}2}(\{|x|\geq C(\eta)/N(t)\})}^{p+2}\Bigr]\\
&\quad \lesssim \eps^{\frac4{p}-(d-2)} \eta^2
\end{align*}
for all $0<t<t_0$.  As $\eta$ can be made arbitrarily small, this proves \eqref{bounded support 2} and hence \eqref{bounded support 1}.

To continue, by \eqref{bounded support 1}, H\"older's inequality, and Sobolev embedding we obtain
\begin{align*}
E(u(t))
&= \int_{B(0,t)} \Bigl(\tfrac12 |\ntx u(t,x)|^2 + \tfrac1{p+2} |u(t,x)|^{p+2}\Bigr) \, dx\\
&\lesssim \Bigl(\|\ntx u(t)\|_{L_x^{\frac{dp}{p+2}}}^2 + \|u(t)\|_{ L_x^{\frac{dp}2}}^{p+2}\Bigr)t^{d-2-\frac4p}\\
& \lesssim_u t^{d-2-\frac4p}
\end{align*}
for all $t\in I$.  In particular, the energy of the solution is finite and converges to zero as the time $t$ approaches
the blowup time $0$. Invoking the conservation of energy, we deduce that $u\equiv 0$.  This contradicts the fact that
$u$ is a blowup solution.
\end{proof}

%
%
%
%

\end{document}